\documentclass[12pt,bezier]{article}
\usepackage{times}
\usepackage{booktabs}
\usepackage{pifont}
\usepackage{floatrow}
\usepackage{caption}
\usepackage{mathrsfs}
\usepackage[fleqn]{amsmath}
\usepackage{amsfonts,amsthm,amssymb,mathrsfs,bbding}
\usepackage{txfonts}
\usepackage{graphics,multicol}
\usepackage{graphicx}
\usepackage{color}

\usepackage{caption}
\usepackage{cite}
\usepackage{latexsym,bm}
\usepackage{indentfirst}
\usepackage{color}
\usepackage[colorlinks=true,anchorcolor=blue,filecolor=blue,linkcolor=blue,urlcolor=blue,citecolor=blue]{hyperref}
\usepackage{extarrows}
\usepackage{cite}
\usepackage{latexsym,bm}
\usepackage{mathtools}
\evensidemargin=\oddsidemargin\topmargin=-1.5cm

\newtheorem{thm}{Theorem}[section]

\newtheorem{prob}{Problem}[section]

\newtheorem{lem}{Lemma}[section]
\newtheorem{cor}{Corollary}[section]

\newtheorem{claim}{Claim}[section]
\theoremstyle{definition}

\addtocounter{section}{0}

\begin{document}
\title{Extremal problems on planar graphs without $k$ edge-disjoint cycles\footnote{Supported by the National Natural Science Foundation of China
(No. 12171066), Anhui Provincial Natural Science Foundation (No. 2108085MA13) and Natural Science Foundation of Guangdong Province (No. 2022A1515011786).}}
\author{{\bf Mingqing Zhai$^{a}$},
~{\bf Muhuo Liu$^{b}$}\thanks{ E-mail addresses: mqzhai@chzu.edu.cn
(Zhai); liumuhuo@163.com (Liu, Corresponding author).}
\\
{\footnotesize $^a$ School of Mathematics and Finance, Chuzhou University, Chuzhou, Anhui 239012, China} \\
{\footnotesize $^b$ Department of Mathematics, South China Agricultural University, Guangzhou, Guangdong, 510642, China}}
\date{}

\date{}
\maketitle
{\flushleft\large\bf Abstract}
In the 1960s, Erd\H{o}s and his cooperators initiated the research of the maximum numbers of edges in a graph or a planar graph on $n$ vertices without $k$ edge-disjoint cycles.
This problem had been solved for $k\leq4$.
As pointed out by Bollob\'{a}s, it is very difficult for general $k$.
Recently, Tait and Tobin [J. Combin. Theory Ser. B, 2017] confirmed a famous conjecture on maximum spectral radius of $n$-vertex planar graphs.
Motivated by the above results,
we consider two extremal problems on planar graphs without $k$ edge-disjoint cycles.
We first determine the maximum number of edges in a planar graph of order $n$ and maximum degree $n-1$ without $k$ edge-disjoint cycles.
Based on this, we then determine the maximum spectral radius as well as
its unique extremal graph over all planar graphs on $n$ vertices
without $k$ edge-disjoint cycles.
Finally, we also discuss several extremal problems for general graphs.

\begin{flushleft}
\end{flushleft}

\section{Introduction}
Let $G$ be a simple graph.
As usual, we denote by $V(G)$ the vertex set and $E(G)$ the edge set.
Let $v(G)$, $e(G)$ and $t(G)$ be the numbers of vertices, edges and triangles in $G$, respectively.
For a vertex $v\in V(G)$,
$N_G(v)$ denotes its neighborhood
and $d_G(v)$ denotes its degree in $G$. A \emph{$k$-vertex} is a vertex of degree $k$.
A \emph{block} is a maximal connected subgraph which contains no cut vertices.
A connected graph $G$ is called a \emph{triangle-cactus},
if every block of $G$ is a triangle.
Given two graphs $G_1$ and $G_2$ with $V(G_1)\cap V(G_2)=\varnothing$,
let $G_1\nabla G_2$ be the \emph{join} of $G_1$ and $G_2$,
which is obtained by joining each vertex of $G_1$ with each of $G_2$.
A \emph{friendship graph} $F_k$ is the graph obtained
from $k$ triangles by sharing exactly one vertex.
We define $F_0$ as an isolated vertex.

Let $g(k)$ denote the least integer such
that every graph with $n$ vertices and $n+g(k)$
edges, contains at least $k$ edge-disjoint cycles,
while $h(k)$ denotes the least integer such that every planar graph with $n$ vertices and $n+h(k)$ edges contains at least $k$ edge-disjoint cycles.
One can easily check that $g(1)=h(1)=0$.
Erd\H{o}s and his cooperators initiated the research of graphs without $k$ edge-disjoint cycles \cite{erd1,erd2,erd3}.
In 1962, Erd\H{o}s and P\'osa \cite{erd2}
showed that $g(2)=4$. Later, Dirac and Erd\H{o}s \cite{erd1} discovered that $h(2)=3$ in 1963, and Moon \cite{moon} identified that $g(3)=10$, $h(3)=7$ and $h(4)=11$ in 1964. In the sequel, the exact value of $g(4)$ was determined by Bollob\'as, who showed $g(4)=18$ in his classic monograph \cite{BO}. Here, we need to point out that the graphs concerned in \cite{erd1,erd2,moon,BO}
permit loops and multiple edges, but the results also hold for simple graphs.
As pointed out by Bollob\'as (see \cite{BO}, P. 121), it is very difficult to determine $g(k)$ for general $k$. In this line, a result due to Erd\H{o}s and P\'osa \cite{erd2} indicates that $g(k)=\Theta(k\log_2^k)$.
This was improved by Bollob\'as (see Theorem 3.6 of \cite{BO}, P. 125), who showed that $\frac{1}{2}k\log_2^k<g(k)<\big(2+o(1)\big)k\log_2^k.$
In contrast with $g(k)$, H\"{a}ggkvist obtained a sharp lower bound for $h(k)$ (see Theorem 3.9 of \cite{BO}, P. 130), he proved that $h(k)\ge 4k-5$ for every $k\ge 2$. Moreover, if every planar graph of order $n$ contains at least $\frac n4$ independent vertices, then $h(k)=4k-5$ for every $k\ge 2$.

In this paper, we first obtain the following edge-extremal result,
which determines the exact value of $h(k)$ for planar graphs with a dominating vertex.

\begin{thm}\label{thm1.1}
Let $G$ be a planar graph of order $n$ and maximum degree $n-1$
without $k$ edge-disjoint cycles.
Then $e(G)\leq n+3k-4$, with equality if and only if $G\cong K_1\nabla H$, where $t(H)=k-1$ and every non-trivial component of $H$ is a triangle-cactus.
\end{thm}

Given a graph $G$,
let $A(G)$ be its adjacency matrix and $\rho(G)$ be the spectral radius of $A(G)$.
Maximizing the spectral radius over a fixed family of graphs is known as
\emph{Brualdi-Solheid problem}.

A variation of Brualdi-Solheid problem, proposed by Nikiforov, asks to maximize the spectral radius over all $n$-vertex $\mathcal{H}$-free graphs,
where $\mathcal{H}$ is a given graph or graph family.
Nikiforov systematically investigated this problem (see \cite{V0}),
and put forward some open questions or conjectures (see \cite{BN,V1,V2}).
Recently, a number of arrestive results on this topic were further presented
(see for example, \cite{CLZ,CFTZ,CS,WKX,CS1,CS2,LIN1,V3,Li,LP}).

As another variation of Brualdi-Solheid problem,
the spectral extremal problem on planar graphs also has a rich history.
In 1990, Cvetkovi\'{c} and Rowlinson \cite{CP} conjectured
that $\rho(G)\leq \rho(K_1\nabla P_{n-1})$ for every outerplanar graph $G$,
with equality if and only if $G\cong K_1\nabla P_{n-1}$.
Later, Boots and Royle \cite{BOOT}, and independently Cao and Vince \cite{CAO},
conjectured that $\rho(G)\leq \rho(K_2\nabla P_{n-2})$
for every planar graph $G$ of order $n\geq9$,
with equality if and only if $G\cong K_2\nabla P_{n-2}$.
Subsequently, many scholars contributed to these two conjectures
(see \cite{CAO,Hong1,Hong2,SHU}).
Ellingham and Zha \cite{EZ} showed that $\rho(G)\leq2+\sqrt{2n-6}$ for every planar graph $G$. Dvo\v{r}\'{a}k and Mohar \cite{DM} proved that $\rho(G)\leq \sqrt{8\Delta-16}+2\sqrt{3}$ for every planar graph $G$ with maximum degree $\Delta\geq2$.
In 2017, Tait and Tobin \cite{Tait1} solved these two conjectures for sufficiently large $n$. Recently, Lin and Ning \cite{LIN2} completely confirmed Cvetkovi\'{c}-Rowlinson conjecture.

Inspired by above edge-extremal and spectral extremal results on planar graphs,
we further consider a spectral extremal problem
for planar graphs without $k$ edge-disjoint cycles.
The extremal graph is uniquely characterized.

\begin{thm}\label{thm1.2}
Let $n\ge 144 k^4$ and $G$ be a planar graph of order $n$ without $k$ edge-disjoint cycles.
Then $\rho(G)\leq \rho(K_1\nabla H)$, with equality if and only if $G\cong K_1\nabla H$, where $H$ is the disjoint union of $F_{k-1}$ and $n-2k$ isolated vertices.
\end{thm}

The case $k=1$ is trivial in Theorems \ref{thm1.1} and \ref{thm1.2}.
Hence, we focus on $k\geq2$ in the subsequent sections.
In section \ref{2}, we prove some useful lemmas.
In Sections \ref{3} and \ref{4}, we give the proofs of Theorems \ref{thm1.1} and \ref{thm1.2} respectively.
In Section \ref{5}, we further investigate several related extremal problems without the assumption graphs being planar.

\section{Preliminaries}\label{2}

In this section, we shall give some lemmas.
Let $\phi(G)$ be the maximum number of edge-disjoint cycles in
a graph $G$.
The following lemma
gives the value of $\phi(G)$ for a graph $G$
obtained by joining a vertex with a triangle-cactus.

\begin{lem}\label{lemma2.1}
If $H$ is a triangle-cactus,
then $K_1\nabla H$ is a planar graph and $\phi(K_1\nabla H)=t(H)$.
\end{lem}

\begin{proof}
Since $H$ is a triangle-cactus, we can easily see that
$K_1\nabla H$ is a planar graph.
Moreover, $H$ has $t(H)$ edge-disjoint triangles
and contains no any other cycles.
Consequently, $\phi(K_1\nabla H)\geq \phi(H)=t(H)$.
In the following, it suffices to show $\phi(K_1\nabla H)\leq t(H)$.
We will show this result by induction on $t(H)$.
If $t(H)=1$, then $H$ itself is a triangle.
It is clear that $\phi(K_1\nabla H)=\phi(K_4)=1$, as desired.
Now suppose that $t(H)=t\geq2$ and $\phi(K_1\nabla H)=s\geq t$.
Let $C^1,C^2,\ldots,C^s$ be $s$ edge-disjoint cycles in $K_1\nabla H$.
Since $H$ is a triangle-cactus with $t$ triangles,
we have $v(H)=2t+1$ (this can be easily shown by induction on $t$).

If there exists a cycle, say $C^1$, in $\{C^1,C^2,\ldots,C^s\}$ such that it is a triangle of $H$, then we define $H'=H-E(C^1)$.
One can see that $H'$ is the union of some triangle-cactus  components and at most two isolated vertices. Since $t(H')=t(H)-1$, we have $\phi(K_1\vee H')=t(H')=t-1$ by induction hypothesis. Now $C^2, C^3,\ldots,C^s$ are $s-1$ edge-disjoint cycles of $K_1\vee H'$ and thus $t-1=\phi(K_1\vee H')\ge s-1$, which implies that $s\leq t$, that is, $\phi(K_1\nabla H)\leq t(H)$.
Otherwise, every cycle of $C^1,C^2,\ldots,C^s$ contains $u^*$ as its vertex, where $u^*\in V(K_1\nabla H)\setminus V(H)$.
It follows that $v(H)=d_{K_1\nabla H}(u^*)\geq 2s$. Recall that $v(H)=2t+1$.
Therefore, we also obtain $s\leq t$, as required.
\end{proof}

\begin{lem}\label{lemma2.2}
Let $G$ be a graph obtained from an isolated vertex $u^*$ and a tree $T$ by adding $t$ edges between $u^*$ and $T$, where $1\leq t\leq v(T)$. Then $\phi(G)= \big\lfloor\frac{t}{2}\big\rfloor$.
\end{lem}
\begin{proof} Since $T$ is a tree, any cycle of $G$ must contains the vertex $u^*$ and thus $\phi(G)\le \big\lfloor\frac{t}{2}\big\rfloor$.
Next, we will show $\phi(G)\ge \big\lfloor\frac{t}{2}\big\rfloor$ by induction on $t$. Clearly, the result holds for $1\le t\le 2$. Next, assume that $t\geq3$.

Note that $N_G(u^*)\subseteq V(T)$.
We may choose $T'$ as a smallest subtree of $T$ such that $N_G(u^*)\subseteq V(T')$.
Let $L(T')$ be the set of leaves of $T'$.
We will see that $L(T')\subseteq N_G(u^*)$.
Otherwise, there exists a leaf $u_0$ of $T'$ with $u_0\notin N_G(u^*)$.
One can observe that $N_G(u^*)\subseteq V(T'-u_0)$,
as we does not remove any vertex in $N_G(u^*)$.
However, $v(T'-u_0)<v(T')$, which contradicts the choice of $T'$.

Recall that $L(T')\subseteq N_G(u^*)\subseteq V(T')$ and $|N_G(u^*)|=t$.
Let $G'$ be the subgraph of $G$ induced by the vertex subset $V(T')\cup \{u^*\}$.
Then $N_{G'}(u^*)=N_{G}(u^*)$, and so $L(T')\subseteq N_{G'}(u^*)$.
In the following, it suffices to show that $\phi(G')\ge \big\lfloor\frac{t}{2}\big\rfloor$, as $\phi(G)\ge \phi(G')$.
Since $L(T')\subseteq N_{G'}(u^*)$, we shall distinguish the subsequent proof into three cases.

\vspace{1mm}
\noindent{{\bf{Case 1.}}} $T'$ contains a path $P=u_1u_2\ldots u_s$ ($s\geq2$) such that $d_{T'}(u_1)=1$, $d_{T'}(u_2)=\cdots=d_{T'}(u_s)=2$ and $u_s\in N_{G'}(u^*)$.

Since $u_1$ is a leaf of $T'$,
we have $u_1\in N_{G'}(u^*)$.
Note that $P$ is a pendant path.
We can define $T''$ as the subtree obtained from $T'$ by deleting $P$,
and define $G''$ as the subgraph of $G'$ induced by $V(T'')\cup \{u^*\}$.
We may assume without loss of generality that $P$ is shortest.
Then $u_i\notin N_{G'}(u^*)$ for each $u_i\in V(P)\setminus \{u_1,u_s\}$.
It follows that $|N_{G''}(u^*)|=|N_{G'}(u^*)\setminus\{u_1,u_s\}|=t-2.$
By the induction hypothesis, $\phi(G'')\ge \lfloor\frac{t-2}{2}\rfloor.$
Since $E(P)\cup \{u^*u_1,u^*u_s\}$ induces a cycle,
we obtain $\phi(G')\geq 1+\phi(G'')\geq \lfloor\frac{t}{2}\rfloor$.

Particularly, if $T'$ itself is a path, we are done by Case 1.
Next, we may assume that $T'$ contains at least one branching vertex
(i.e., a vertex of degree greater than two).

\vspace{1mm}
\noindent{{\bf{Case 2.}}} $T'$ is a starlike tree.

Now $T'$ has a unique branching vertex, say $v$.
For $i\in\{1,2\}$,
let $u_i$ be a leaf of $T'$ and $P_{u_i}$ be the pendant path from $v$ to $u_i$ in $T'$.
Then by Case 1, we may assume that every 2-vertex in $P_{u_1}$ and $P_{u_2}$ does not
belong to $N_{G'}(u^*)$.
We now define $T''$ as the subtree obtained from $T'$ by deleting two pendant paths $P_{u_1}$ and $P_{u_2}$ (where $v$ is reserved),
and define $G''$ as the subgraph of $G'$ induced by $V(T'')\cup \{u^*\}$.
Note that $u_1,u_2\in L(T')\subseteq N_{G'}(u^*)$. Then $|N_{G''}(u^*)|=|N_{G'}(u^*)\setminus\{u_1,u_2\}|=t-2,$
and $E(P_{u_1})\cup E(P_{u_2})\cup \{u^*u_1,u^*u_2\}$ induces a cycle.
By the induction hypothesis,
we have $\phi(G')\geq 1+\phi(G'')\geq \lfloor\frac{t}{2}\rfloor$.

\vspace{1mm}
\noindent{{\bf{Case 3.}}} $T'$ contains at least two branching vertices.

Now we choose two branching vertices $v$ and $w$ of $T'$ such that the distance between $v$ and $w$ is as large as possible.
Then there are at least two pendant paths $P_{u_1}$ and $P_{u_2}$ attaching at $v$,
where $u_1,u_2$ are two leaves.
Note that $u_1,u_2\in N_{G'}(u^*)$, and by Case 1, we may assume that every 2-vertex in $P_{u_1}$ and $P_{u_2}$ does not belong to $N_{G'}(u^*)$.
Define $T''$ and $G''$ similarly as in Case 2.
One can also get $\phi(G')\geq 1+\phi(G'')\geq \lfloor\frac{t}{2}\rfloor$.
\end{proof}

From Lemma \ref{lemma2.2}, we can easily obtain
the exact value of $\phi(G)$ for a graph $G$
obtained by joining a vertex with a tree.

\begin{cor}\label{coro2.1}
Let $T$ be a tree. Then
$\phi(K_1\nabla T)=\big\lfloor\frac{v(T)}{2}\big\rfloor$.
\end{cor}

Furthermore, we can deduce a result on general 2-connected graphs instead of $K_1\nabla T$.

\begin{cor}\label{coro2.2}
Let $G$ be a 2-connected graph.
Then $\phi(G)\ge\big\lfloor\frac{\Delta(G)}{2}\big\rfloor$, where $\Delta(G)$ is the maximum degree of $G$.
\end{cor}

\begin{proof} Let $u^*$ be a vertex of $G$ with $d_G(u^*)=\Delta(G)$. Since $G$ is 2-connected, $G-u^*$ is connected and thus $G-u^*$ contains a spanning tree $T$ which contains $\Delta(G)$ neighbors of $u^*$. Now let $G'$ be the subgraph of $G$ induced
by $E(T)$ and all edges incident to $u^*$.
Then $\phi(G')=\big\lfloor\frac{\Delta(G)}{2}\big\rfloor$ by Lemma \ref{lemma2.2}.
Since $G'\subseteq G$, we have $\phi(G)\geq\phi(G')=\big\lfloor\frac{\Delta(G)}{2}\big\rfloor$.
\end{proof}

\section{Proof of Theorem \ref{thm1.1}}\label{3}

In this section, we shall determine the maximum number of edges in
a planar graph of order $n$ and maximum degree $n-1$ without $k$ edge-disjoint cycles.
For the reason of conciseness, we decompose the proof into three claims.

\begin{claim}\label{claim3.1}
Let $k\geq2$ and $G$ be a planar graph with $\phi(G)\leq k-1$.
If $G$ contains a dominating vertex $u^*$ and $G-u^*$ is triangle-free,
then $e(G-u^*)\le 3k-4$.
\end{claim}

\begin{proof}
We can observe that $G-u^*$ has an outerplane embedding,
more precisely, all vertices of $G-u^*$ lie on its outer face
(otherwise, $G-u^*$ has either a $K_{2,3}$-minor or a $K_4$-minor,
correspondingly, $G$ contains a $K_{3,3}$-minor or a $K_5$-minor,
a contradiction).
Assume that $G-u^*$ contains $q$ components $G_1,G_2,\dots,G_q$.
If $G-u^*$ is acyclic, then
by Corollary \ref{coro2.1}, we have
\begin{align}\label{al001}
\phi(G)=\sum_{i=1}^{q}\phi(K_1\nabla G_i)=\sum_{i=1}^{q}\left\lfloor\frac{v(G_i)}2\right\rfloor\geq
\sum_{i=1}^{q}\frac{e(G_i)}2=\frac{e(G-u^*)}2.
\end{align}
Note that $\phi(G)\leq k-1$. Thus we obtain $e(G-u^*)\leq 2k-2\leq 3k-4$, as desired.

In the following, we assume that $G-u^*$ is not a forest,
then it has at least one inner face.
Suppose that $G-u^*$ contains exactly $t$ inner faces $F_1,F_2,\dots,F_t$, and exactly $s$ edges $e_1,e_2,\ldots,e_s$ incident with two inner faces, where $t\geq1$ and $s\geq0$.

We shall first prove that
$s\le t-1$.
We now define a bipartite graph $H$.
Let $V(H)=\{F_{i}^*: i=1,2,\ldots,t\}\cup \{e_{j}^*: j=1,2,\ldots,s\}$.
A vertex $F_{i}^*$ is adjacent to a vertex $e_{j}^*$ in $H$
if and only if $e_j$ lies on $F_i$ in $G-u^*$.
Then, $H$ is acyclic
(otherwise, we may assume without loss of generality that $F^*_1 e^*_1 F^*_2 e^*_2\cdots F^*_re^*_rF^*_1$ is a cycle. Then one endpoint of $e_1$ cannot lie on the outer face of $G-u^*$, a contradiction).
Now we have $e(H)\leq v(H)-1$.
Note that $v(H)=s+t$ and $e(H)=2s$.
It follows that $s\le t-1$.

Let $G'$ be the subgraph of $G-u^*$ consisting of vertices and edges incident with inner faces of $G-u^*$. Then
\begin{align*}
e(G-u^*)\ge e(G')=\sum_{i=1}^{t}e(F_i)-s\ge 4t-(t-1)=3t+1,
\end{align*}
as $G-u^*$ is triangle-free and $s\leq t-1$. Thus,
\begin{align}\label{a9}
t\le \frac{1}{3}e(G')-\frac{1}{3}\le \frac{1}{3}e(G-u^*)-\frac{1}{3}.
\end{align}

Note that every vertex $e_j^*$ is of degree two in $H$.
Since $H$ is acyclic, there exists a vertex of degree at most one, say $F_{i_1}^*$, in $H$. This implies that there exists an edge $e_1'$ lying on $F_{i_1}$ and the outer face of $G-u^*$, that is, $e_1'\notin\{e_1,e_2,\ldots,e_s\}$.
Let $G'_1:=G-u^*-e_1'$. Clearly, $G'_1$ contains $t-1$ inner faces.
If $t\geq2$, then we repeat this step to obtain a sequence of graphs $G'_1,G'_2,\ldots,G'_t$ such that $G'_i=G'_{i-1}-e_i'$ and it contains
exactly $t-i$ inner faces for each $i\in\{2,3,\ldots,t\}$.
Consequently, $e(G'_t)=e(G-u^*)-t$ and $G'_t$ is acyclic.
By (\ref{al001}), we have
\begin{align*}
k-1\ge \phi(G)\ge \phi(K_1\nabla G'_t)\ge \frac{e(G'_t)}{2}=\frac{e(G-u^*)-t}{2}.
\end{align*}
It follows that $e(G-u^*)\le 2k-2+t$.
Combining (\ref{a9}), we obtain $e(G-u^*)\le 3k-\frac{7}{2}$, that is, $e(G-u^*)\le 3k-4$, as claimed.
\end{proof}

\begin{claim}\label{claim3.2}
Let $k\geq 1$ and $G$ be a planar graph with $\phi(G)\leq k-1$.
If $G$ contains a dominating vertex $u^*$,
then $e(G-u^*)\le 3k-3$,
where the equality implies that $G-u^*$ can be decomposed into exactly $k-1$ edge-disjoint triangles.
\end{claim}

\begin{proof}
If $k=1$, then $G$ is a star and the claim trivially holds.
We now assume that $k\geq2$.
By Claim \ref{claim3.1},
it remains to consider the case $t(G-u^*)\geq1$.
Suppose that $G-u^*$ contains a triangle $C^{(1)}$.
Then we define $G_1':=G-u^*-E(C^{(1)})$.
If $t(G_1')\geq1$, then we repeat this step to
obtain a sequence of graphs $G'_1,G'_2,\ldots,G'_q$ such that

(i) $G'_i=G'_{i-1}-E(C^{(i)})$ for each $i\in\{2,3,\ldots,q\}$, where $C^{(i)}$ is a triangle in $G'_{i-1}$;

(ii) $G'_q$ contains no triangles.\\
Since $C^{(1)},C^{(2)},\ldots,C^{(q)}$ are edge-disjoint,
we have $q\leq\phi(G)\le k-1$ and $\phi(K_1\nabla G'_q)\le k-1-q$.
Note that $G'_q$ is triangle-free.
If $q\le k-2$, then $e(G'_q)\le 3(k-q)-4$ by Claim \ref{claim3.1}.
Therefore, $e(G-u^*)=e(G'_q)+3q\le 3k-4$, as desired.
If $q=k-1$, then $\phi(K_1\nabla G'_q)=0$.
Thus $e(G'_q)=0$ and $e(G-u^*)=e(G'_q)+3q=3k-3$.
Now $G-u^*$ can be decomposed into exactly $k-1$ edge-disjoint triangles
$C^{(1)},C^{(2)},\ldots,C^{(k-1)}$.
\end{proof}

\begin{claim}\label{claim3.3}
Let $k\geq 1$ and $G$ be a planar graph with $\phi(G)\leq k-1$.
If $G$ contains a dominating vertex $u^*$ and $e(G-u^*)=3k-3$,
then every non-trivial component of $G-u^*$ is a triangle-cactus
and $\phi(G)=t(G-u^*)=k-1$.

\end{claim}

\begin{figure}[!ht]
	\centering
	\includegraphics[width=0.4\textwidth]{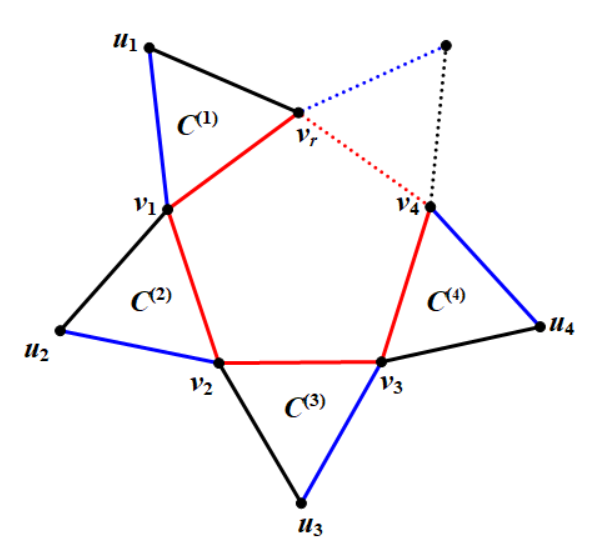}
	\caption{Local structure of $G-u^*$. }{\label{figu.5}}
\end{figure}

\begin{proof}
We only need  to  consider the case $k\geq2$.
By   Claim \ref{claim3.2}, if $e(G-u^*)=3k-3$, then $G-u^*$ contains $k-1$ edge-disjoint triangles $C^{(1)},C^{(2)},\dots,C^{(k-1)}$ such that $E(G-u^*)=\bigcup_{i=1}^r E(C^{(i)})$.

Assume that $G-u^*$ contains exactly $p$ vertices $v_1,v_2,\ldots,v_p$, each of which
belongs to at least two triangles.
We now define a bipartite graph $H^*$ with
vertex set $V(H^*)=\{u_i^*: i=1,2,\ldots,k-1\}\cup\{v_j^*: j=1,2,\ldots,p\}$.
A vertex $u_i^*$ is adjacent to a vertex $v_j^*$ in $H^*$ if and only if $v_j$
lies on the triangle $C^{(i)}$ in $G-u^*.$

We shall prove that $H^*$ is acyclic.
Suppose to the contrary that $u_1^*v^*_1u_2^*v^*_2\cdots u_r^*v^*_ru_1^*$ is
a cycle in $H^*$. Then $r\leq \min\{p,k-1\}$.
Now, we can find $k$ edge-disjoint cycles, which consist of $k-r-1$ triangles
$C^{(r+1)},C^{(r+2)},\dots,C^{(k-1)}$, $r$ triangles  $u^*u_1v_1u^*,\dots,u^*u_rv_ru^*$, and an $r$-cycle $v_1v_2\cdots v_rv_1$ (see Fig. \ref{figu.5}).
This contradicts the fact $\phi(G)\leq k-1$.

Since $H^*$ is acyclic and $E(G-u^*)=\bigcup_{i=1}^r E(C^{(i)})$,
one can observe that every component of $G-u^*$ is either an isolated vertex or a triangle-cactus, and thus $t(G-u^*)=k-1$.
By Lemma \ref{lemma2.1}, $\phi(G)=t(G-u^*)=k-1$.
The proof is complete.
\end{proof}

Combining Claim \ref{claim3.2} and Claim \ref{claim3.3},
we complete the proof of Theorem \ref{thm1.1}.
Theorem \ref{thm1.1}, together with Claims \ref{claim3.1}-\ref{claim3.3}, will be used in the subsequent section to characterize
the structure of spectral extremal graphs.

\section{Proof of Theorem \ref{thm1.2}}\label{4}

This section will be dedicated to the proof of Theorem \ref{thm1.2}.
Let $G_{n,k}$ be an extremal graph with maximum spectral radius over all planar graphs of order $n$ without $k$ edge-disjoint cycles. Then
$G_{n,k}$ must be connected. To see this, suppose to the contrary that $G_{n,k}$ is disconnected. Let $G_1, G_2$ be two components of $G_{n,k}$ so that $\rho(G_{n,k})=\rho(G_1)$, and let $G$ be obtained from $G_{n,k}$ by adding an edge between $G_1$ and $G_2$. Since $G_1$ is a proper subgraph of $G$, $\rho(G)>\rho(G_{n,k})$ and $G$ is also a planar graph with $\phi(G)\leq k-1$, contradicting the choice of $G_{n,k}$. Therefore, we may assume that $X$ is the Perron vector of $G_{n,k}$ with its entry $x_u$ corresponding to the vertex $u\in V(G_{n,k})$.
Note that $K_{1,n-1}$ is a planar graph of order $n$ without $k$ edge-disjoint cycles.
By the choice of $G_{n,k}$, we have \begin{align}\label{liu1e}\rho(G_{n,k})\ge
\rho(K_{1,n-1})=\sqrt{n-1}.\end{align} Next, we give two lemmas.

\begin{lem}\label{lemma4.1}
Let $n\ge 16k^2$ and $u^*$ be a vertex in $V(G_{n,k})$ with $x_{u^*}=\max_{u\in V(G_{n,k})}x_u$. Then $u^*$ is a dominating vertex.
\end{lem}

\begin{proof}
Suppose that $H$ is an arbitrary block of $G_{n,k}$.
It suffices to prove that $u^*\in V(H)$ and $u^*$ is a dominating vertex of $H$.
 We  distinguish the following two cases.

\vspace{1mm}
\noindent{{\bf{Case 1.}}} $H\cong K_2$.

We only need to show $u^*\in V(H)$.
Let $V(H)=\{u_1,u_2\}$ and suppose to the contrary that $u^*\notin V(H)$.
We may assume that $d_{G_{n,k}}(u_2,u^*)\geq d_{G_{n,k}}(u_1,u^*)$,
where $d_{G_{n,k}}(u_2,u^*)$ denotes
the distance between $u_2$ and $u^*$ in $G_{n,k}.$
Since $u_1u_2$ is a cut edge of $G_{n,k}$,
there exists no block which contains $u_2$ and $u^*$ simultaneously.
Now we define $G=G_{n,k}-u_1u_2+u^*u_2$ (see Fig. \ref{figu.3}).
Clearly, $\phi(G)=\phi(G_{n,k})$ and $G$ is also a connected planar graph.
Moreover,
\begin{align*}
\rho(G)-\rho(G_{n,k})\ge X^T(A(G)-A(G_{n,k}))X=2(x_{u^*}x_{u_2}-x_{u_1}x_{u_2})\ge 0.
\end{align*}
\begin{figure}[!ht]
	\centering
	\includegraphics[width=0.6\textwidth]{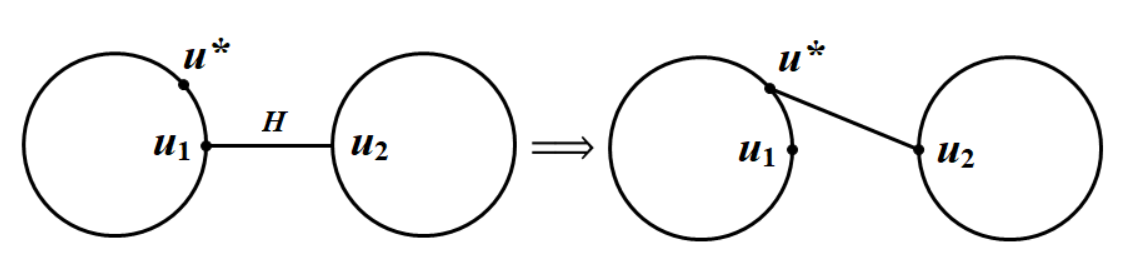}
	\caption{Graphs $G_{n,k}$ and $G$. }{\label{figu.3}}
\end{figure}
It follows that $\rho(G)\ge \rho(G_{n,k})$.
If $\rho(G)=\rho(G_{n,k})$,
then $X$ is the Perron vector of $G$, and thus
$\sum_{u\in N_{G}(u^*)}x_u=\rho(G)x_{u^*}=\rho(G_{n,k})x_{u^*}=\sum_{u\in N_{G_{n,k}}(u^*)}x_u.$ This follows that
$x_{u_2}=0$,
which contradicts the fact that the Perron vector $X$ is a positive vector.
Thus $\rho(G)>\rho(G_{n,k})$, a contradiction.
Therefore, $u^*\in V(H)$.

\vspace{1mm}
\noindent{{\bf{Case 2.}}} $H\not\cong K_2$, that is, $v(H)\ge3$.

Note that $H$ is a block and $H\not\cong K_2$. Then $H$ is 2-connected.
Moreover,
if $u^*\notin V(H)$, then $u^*$ has at most one neighbor in $V(H)$.
We can similarly define a new planar graph $G'$ with
$\phi(G')=\phi(G_{n,k})$ while $\rho(G')>\rho(G_{n,k})$ (see Fig. \ref{figu.4}).
Therefore, $u^*\in V(H)$.

\begin{figure}[!ht]
	\centering
	\includegraphics[width=0.45\textwidth]{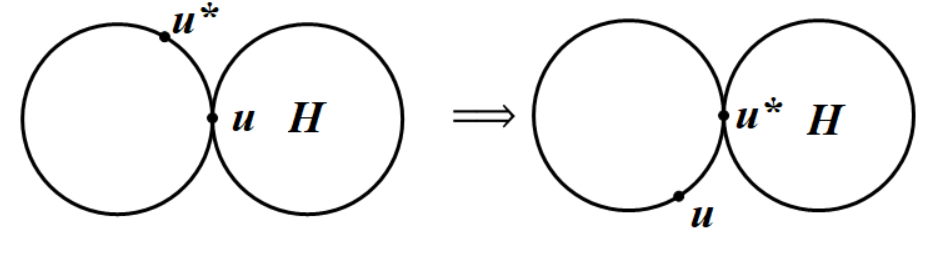}
	\caption{Graphs $G_{n,k}$ and $G'$. }{\label{figu.4}}
\end{figure}

Combining with  Case 1,
we can see that every block of $G_{n,k}$ contains the vertex $u^*$.
This implies that $u^*$ is the only possible cut vertex of $G_{n,k}$.
Since $\phi(H)\leq\phi(G_{n,k})\le k-1$, by Corollary \ref{coro2.2}, we have
\begin{align}\label{al03}
\Delta(H)\le 2k-1.
\end{align}
Now let $\rho=\rho(G_{n,k})$ and $v^*\in V(H)$
with $x_{v^*}=\max_{u\in V(H)\setminus \{u^*\}}x_u$.
Note that $N_{G_{n,k}}(v^*)=N_H(v^*).$
By (\ref{al03}), we have
$\rho x_{v^*}=\sum_{u\in N_H(v^*)}x_u\le (2k-2)x_{v^*}+x_{u^*}.$
From (\ref{liu1e}), it follows that $\rho>2k-2$ and thus
\begin{align}\label{al04}
0< x_{v^*}\le \frac{x_{u^*}}{\rho-2k+2}.
\end{align}

In the following, we show that $u^*$ is a dominating vertex of $H$.
If not, then we can find a vertex $v\in V(H)$ such that it is not a neighbor of $u^*$.
Now we define $G''=G_{n,k}-\{vv_i: v_i\in N_H(v)\}+vu^*$.
Then $\phi(G'')\le \phi(G_{n,k})\le k-1$ and $G''$ is also a planar graph.
Combining (\ref{al03}) and (\ref{al04}), we have
$$\rho(G'')-\rho(G_{n,k})\ge 2x_v\,\Big(x_{u^*}-\sum_{v_i\in N_H(v)}x_{v_i}\Big)
\ge 2x_v\big(x_{u^*}-(2k-1)x_{v^*}\big)
\ge2 x_v x_{u^*} \Big(1-\frac{2k-1}{\rho-2k+2}\Big).
$$
Since $n\ge 16k^2$ and $\rho(G_{n,k})\ge \sqrt{n-1}$,
we have $\rho(G_{n,k})>4k-3$,
which implies that $\rho(G'')>\rho(G_{n,k})$, a contradiction.
This completes the proof.
\end{proof}

\begin{lem}\label{lemma4.2}
Let $n \ge 144  k^4$ and $\mathcal{G}_{n,k}$ be the family of graphs with maximum number of edges over all planar graphs of order $n$ and maximum degree $n-1$ without $k$ edge-disjoint cycles. Then $G_{n,k}\in \mathcal{G}_{n,k}$.
\end{lem}

\begin{proof}
Let $u$ be an arbitrary vertex in $V(G_{n,k})\setminus \{u^*\}$ and $\rho=\rho(G_{n,k})$.
Since $u^*$ is a dominating vertex by Lemma \ref{lemma4.1},
we have $\rho x_u\geq x_{u^*}$, that is,
\begin{align}\label{al05}
x_u\ge \frac{x_{u^*}}{\rho}.
\end{align}
Moreover, by (\ref{al04}) we have
\begin{align}\label{al06}
x_u\le \frac{x_{u^*}}{\rho-2k+2}.
\end{align}
Now suppose to the contrary that $G_{n,k}\notin \mathcal{G}_{n,k}$.
We choose an arbitrary graph $G\in \mathcal{G}_{n,k}$.
Then $e(G_{n,k})\le e(G)-1$.
We may assume that $V(G)=V(G_{n,k})$ and $u^*$ is also a dominating vertex of $G$.
Let $H^*=G_{n,k}-u^*$ and $H=G-u^*$. Then
\begin{align}\label{al07}
e(H^*)\le e(H)-1.
\end{align}
Furthermore,
\begin{align}\label{al08}
\rho(G)-\rho(G_{n,k})\geq X^T\,\big(A(G)-A(G_{n,k})\big)\,X=
2\left(\sum_{uv\in E(H)}x_ux_v-\sum_{uv\in E(H^*)}x_ux_v\right).
\end{align}
By (\ref{al05}), (\ref{al06}) and (\ref{al07}), we obtain
\begin{eqnarray}\label{al09}
\sum_{uv\in E(H)}x_ux_v-\sum_{uv\in E(H^*)}x_ux_v
&\geq&e(H)\,\frac{x_{u^*}^2}{\rho^2}- e(H^*)\,\frac{x_{u^*}^2}{(\rho-2k+2)^2}\nonumber\\
&\geq& x_{u^*}^2\left(\frac{e(H)}{\rho^2}-\frac{e(H)-1}{(\rho-2k+2)^2}\right).
\end{eqnarray}
By Claim \ref{claim3.2} and Claim \ref{claim3.3}, we know that $e(H)=3k-3$.
Since $n\ge 144k^4$, we have $\rho(G_{n,k})\ge \sqrt{n-1}>12(k-1)^2$.
Thus by (\ref{al08}) and (\ref{al09}), we get that $\rho(G)-\rho(G_{n,k})\geq 2\big(\sum_{uv\in E(H)}x_ux_v-\sum_{uv\in E(H^*)}x_ux_v\big)>0$,
a contradiction. This completes the proof.
\end{proof}

In the following, we shall give the proof of Theorem \ref{thm1.2}.

\vspace{1mm}
\noindent{\bf Proof of Theorem \ref{thm1.2}.}
By Lemma \ref{lemma4.2} and Theorem \ref{thm1.1},
we have $G_{n,k}\cong K_1\nabla H$, where $t(H)=k-1$ and every non-trivial component of $H$ is a triangle-cactus.
Let $u^*,v^*\in V(G_{n,k})$ with $x_{u^*}=\max_{u\in V(G_{n,k})}x_u$ and $x_{v^*}=\max_{u\in V(G_{n,k})\setminus\{u^*\}}x_u$.
By Lemma \ref{lemma4.1}, $u^*$ is a dominating vertex of $G_{n,k}$.

Now it suffices to show that $G_{n,k}-u^*$
is the disjoint union of $F_{k-1}$ and $n-2k$ isolated vertices.
Suppose to the contrary, then there exists a triangle $C=u_0u_1u_2u_0$ in
$G_{n,k}-u^*$ such that $v^*\notin V(C)$.
We can see that $v^*$ has at most one neighbor in $V(C)$, as  every non-trivial component of $G_{n,k}-u^*$ is a triangle-cactus.
Now we may assume without loss of generality that $u_1, u_2\not\in N_{G_{n,k}}(v^*)$.
Let $G$ be the graph obtained from $G_{n,k}$ by replacing the edges
$u_0u_1$ and $u_0u_2$ with the edges $v^*u_1$ and $v^*u_2$.
Since every non-trivial component of $G_{n,k}-u^*$ is a triangle-cactus,
it is clear that $G-u^*$ has the same property.
By Lemma \ref{lemma2.1}, $G$ is a planar graph with $\phi(G)=t(G-u^*)=t(G_{n,k}-u^*)=k-1$.
However,
\begin{align*}
\rho(G)-\rho(G_{n,k})\geq X^T\,\big(A(G)-A(G_{n,k})\big)\,X=2(x_{v^*}-x_{u_0})(x_{u_1}+x_{u_2})
\geq0.
\end{align*}
Furthermore, we can get that $\rho(G)>\rho(G_{n,k})$ by a similar discussion as in the proof of Lemma \ref{lemma4.1},
a contradiction.
This completes the proof.
\hfill{\rule{4pt}{8pt}}

\section{Conclusion remarks}\label{5}

In this section, we consider extremal problems without the assumption graphs being planar.
In the classic monograph due to Bollob\'{a}s \cite{BO},
we can find the following edge-extremal problem.

\begin{prob}\label{prob5.1}
What is the maximum number of edges in a graph of order $n$
without $k$ edge-disjoint cycles, or equivalently, what is the exact value of $g(k)$?
\end{prob}

As pointed out by Bollob\'{a}s, Problem \ref{prob5.1} is rather difficult.
We now provide a weak version of Problem \ref{prob5.1} as follows.

\begin{prob}\label{prob5.2}
What is the maximum number of edges in a graph of order $n$ and maximum degree $n-1$
without $k$ edge-disjoint cycles?
\end{prob}

In the following,
we pay our attention  to  a spectral version of Problem \ref{prob5.1},
which was proposed in \cite{LIN}.

\begin{prob}\label{prob5.3}
What is the maximum spectral radius of a graph of order $n$
without $k$ edge-disjoint cycles?
\end{prob}

Let $\mathcal{G}'_{n,k}$ and $\mathcal{G}''_{n,k}$ be the families of extremal graphs for Problem \ref{prob5.2} and Problem \ref{prob5.3}, respectively.
Then we have the following result.

\begin{thm}\label{thm5.1}
If $n\ge 64k^6$, then $\mathcal{G}''_{n,k}\subseteq \mathcal{G}'_{n,k}$.
\end{thm}

\begin{proof} Let $G_{n,k}$ be an arbitrary graph in $\mathcal{G}''_{n,k}$
and $X$ be the Perron vector of  $G_{n,k}$ (one can easily check that $G_{n,k}$ must be connected).   If $k=1$, then $G_{n,k}$ is acyclic, and hence $G_{n,1}\cong K_{1,n-1}$, as desired.  Thus, we suppose that $k\ge 2$ in the following.
Assume that $u^*\in V(G_{n,k})$ with $x_{u^*}=\max_{u\in V(G_{n,k})}x_u.$
Then by the proof of Lemma \ref{lemma4.1}, we can see that
Lemma \ref{lemma4.1} also holds for general graphs instead of planar graphs,
that is, $u^*$ is a dominating vertex of $G_{n,k}$.

Now suppose to the contrary that $G_{n,k}\notin \mathcal{G}'_{n,k}$.
We choose a graph $G\in \mathcal{G}'_{n,k}$.
Then $e(G_{n,k})\leq e(G)-1$.
We may assume that $V(G)=V(G_{n,k})$ and $u^*$ is also a dominating
vertex of $G$. Let $H^*=G_{n,k}-u^*$ and $H=G-u^*$. Then, we similarly
have inequalities (\ref{al05})-(\ref{al09}).

Assume that $H$ contains $r$ non-trivial connected components $H_1,H_2,\dots,H_r$.
Then $e(H)=\sum_{i=1}^{r}e(H_i)$ and
\begin{align}\label{al010}
\phi(G)=\sum_{i=1}^{r}\phi(K_1\nabla H_i)\le k-1.
\end{align}
Since every $K_1\nabla H_i$ is 2-connected, we have
$\Delta(K_1\nabla H_i)\le 2\phi(K_1\nabla H_i)+1$ by Corollary \ref{coro2.2}.
This implies that $v(H_i)\le 2\phi(K_1\nabla H_i)+1$,
and thus
\begin{align}\label{al011}
e(H_i)\le{v(H_i) \choose 2}\le \phi(K_1\nabla H_i)(2\phi(K_1\nabla H_i)+1).
\end{align}
Combining  with  (\ref{al010}) and (\ref{al011}), we obtain
\begin{align*}
e(H)=\sum_{i=1}^{r}e(H_i)\leq \sum_{i=1}^{r}\phi(K_1\nabla H_i)+2\sum_{i=1}^{r}\phi^2(K_1\nabla H_i)
\leq(k-1)+2(k-1)^2,
\end{align*}
which gives $e(H)<2k^2.$
Moreover, by (\ref{al08}) and (\ref{al09}), we have
\begin{eqnarray*}
\rho(G')-\rho(G_{n,k})
&\ge& 2x_{u^*}^2\,\left(\frac{e(H)}{\rho^2}-\frac{e(H)-1}{(\rho-2k+2)^2}\right) \\
&>& \frac{2x_{u^*}^2}{\rho(\rho-2k+2)^2}\Big(\rho-4(k-1)\,e(H)\Big)
\end{eqnarray*}
where $\rho=\rho(G_{n,k}).$
Note that $ n\ge 64k^6$. Then $\rho\geq\sqrt{n-1} >8k^3-1$.
We can check that $\rho(G)>\rho(G_{n,k})$,
which contradicts the choice of $G_{n,k}$.
This completes the proof.
\end{proof}

By Theorem \ref{thm5.1}, we can determine $\mathcal{G}''_{n,k}$
for small $k$.

\begin{thm}\label{thm5.2}
Let $k\in \{1,2,3\}$ and $n\geq64k^6$.
Then the extremal graph in Problem \ref{prob5.3}
is the same as  that  in Theorem \ref{thm1.2}.
\end{thm}

\begin{proof}
Let $G_{n,k}$ be an arbitrary graph in Problem \ref{prob5.3}.
Then $G_{n,k}\in \mathcal{G}'_{n,k}$ by Theorem \ref{thm5.1}.
This implies that $G_{n,k}$ contains a dominating vertex $u^*$.
If $k=1$, then    $G_{n,1}\cong K_{1,n-1}$, as desired.
If $k=2$, then $\phi(G_{n,k})\leq1$,
and thus $G_{n,2}-u^*$ has at most one non-trivial component $H$.
By Corollary \ref{coro2.2},
\begin{align*}
\Delta(K_1\nabla H)\leq 2\phi(K_1\nabla H)+1\leq2\phi(G_{n,2})+1\leq3.
\end{align*}
Therefore, $v(H)=\Delta(K_1\nabla H)\leq3$.
Since $G_{n,2}$ is edge-maximal, we can see that $H\cong K_3$.
Hence, $G_{n,2}-u^*$ is the disjoint union of a triangle $F_1$ and
$n-4$ isolated vertices.

If $k=3$, then $\phi(G_{n,k})\leq2$,
and thus $G_{n,3}-u^*$ has at most two non-trivial components.
If $G_{n,3}-u^*$ has exactly two non-trivial components,
then using the same argument as above, we get that
both components are triangles.
Now let $G$ be the graph obtained from $G_{n,3}$ by replacing these two vertex-disjoint triangles with $K_1\cup F_2$.
It is easy to see that $\phi(G)=2$ and $\rho(G)>\rho(G_{n,2})$, a contradiction.
Therefore, $G_{n,3}-u^*$ has at most one non-trivial component $H$.
By Corollary \ref{coro2.2},
\begin{align*}
\Delta(K_1\nabla H)\leq 2\phi(K_1\nabla H)+1\leq2\phi(G_{n,3})+1\leq5.
\end{align*}
Thus $v(H)=\Delta(K_1\nabla H)\leq5$.
Moreover, since $\phi(K_1\nabla F_2)=2$, we have $e(H)\geq e(F_2)=6$
as $G_{n,3}$ is edge-maximal by Theorem \ref{thm5.1}.
If $v(H)\leq4$, then $H\cong K_4$. But now $\phi(G_{n,3})=\phi(K_5)=3$
($K_5$ can be decomposed into two triangles and a quadrilateral),
a contradiction. Therefore, $v(H)=5$.
Recall that $e(H)\geq6$. If $\phi(H)=1$,
then $H$ contains one of the following bicyclic spanning subgraphs (see Fig. \ref{figu.6}).
\begin{figure}[!ht]
	\centering
	\includegraphics[width=0.6\textwidth]{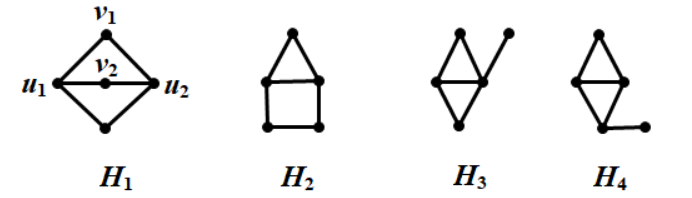}
	\caption{Graphs $H_1,H_2,H_3$ and $H_4$. }{\label{figu.6}}
\end{figure}
For $i\in \{2,3,4\}$, we can observe that $H_i$ contains a triangle and two independent edges; thus $\phi(K_1\nabla H_i)\geq3$, a contradiction.
Therefore, $H$ contains $H_1$ as a spanning subgraph.
Furthermore, $H\cong H_1$, as $H_1$ plus one edge contains a triangle and two independent edges. Put $G=G_{n,3}-u_1v_1-u_1v_2+u_1u_2+v_1v_2$.
In other words, $G$ is the graph obtained from $G_{n,3}$ by replacing $H_1$ with $F_2$.
Let $X$ be the Perron vector of $G_{n,3}$.
We know that $X$ is a positive vector.
Since $x_{u_1}=x_{u_2}$ and $x_{v_1}=x_{v_2}$ by symmetry, we get that
\begin{align*}\rho(G)-\rho(G_{n,3})\ge 2(x_{u_1}x_{u_2}+x_{v_1}x_{v_2}-x_{u_1}x_{v_1}-x_{u_1}x_{v_2})
=2(x_{u_1}-x_{v_1})^2\geq0.
\end{align*}
By the choice of $G_{n,3}$, we have $\rho(G)=\rho(G_{n,3})$ and so $x_{u_1}=x_{v_1}$.
This implies that $x_{u^*}+3x_{v_1}= \rho(G_{n,3})x_{u_1}=\rho(G_{n,3}) x_{v_1}=x_{u^*}+2x_{u_1}$, contradicting $x_{v_1}>0$.

Therefore, $\phi(H)=2$. Note that $v(H)=5$ and $e(H)\geq6$.
Then $H$ must contain $F_2$ as a spanning subgraph.
Since $\phi(G_{n,3})\leq2$, we know that $H$ minus two independent edges is acyclic. This implies that $H\cong F_2$, as desired.
\end{proof}

Now assume that $k\geq4$.
One can check that $\phi(K_1\nabla K_{2,2k-3})=k-1$
(where $K_1\nabla K_{2,2k-3}$ contains $k-1$ edge-disjoint cycles,
which consist of a triangle and $k-2$ quadrilaterals).
However, $K_1\nabla K_{2,2k-3}$ is not a planar graph, since it contains a $K_{3,3}$-minor.
Moreover, $K_1\nabla (K_{2,2k-3}\cup (n-2k)K_1)$ has $n+4k-7$ ($>n+3k-4$) edges.
Combining Theorem \ref{thm5.1}, we can see that the extremal graph in Problem \ref{prob5.3} is no longer the same as the extremal graph in Theorem \ref{thm1.2}
for $k\geq4$.

Theorem \ref{thm5.1} indicates that
Problem \ref{prob5.3} can be reduced to Problem \ref{prob5.2}.
Furthermore, one can see that Problem \ref{prob5.2} is closely related to
the following problem.

\begin{prob}\label{prob5.4}
Let $\mathcal{H}_s$ be a family of connected graphs such that
$\phi(K_1\nabla H)=s$ for each graph $H\in\mathcal{H}_s$.
Determine $\max_{H\in\mathcal{H}_s}\frac{e(H)}{s}.$
\end{prob}

By Corollary \ref{coro2.2}, $v(H)=\Delta(K_1\nabla H)\leq 2s+1$
for each graph $H\in\mathcal{H}_s$.
Therefore, we immediately obtain a trivial bound
$\max_{H\in\mathcal{H}_s}\frac{e(H)}{s}\leq \frac1s{2s+1\choose 2}=2s+1.$
Assume that there exists a graph $H\in\mathcal{H}_s$
such that $H$ is $d$-regular with girth $g$.
Then $2e(H)=d\cdot v(H)$.
Suppose that among $s$ edge-disjoint cycles of $K_1\nabla H$, $q$ of them  contains the dominating vertex $u^*$, where $q\leq s$. Since each of these $q$ cycles contains at least one edge of $H$, we have $s\le q+\frac1g\big(e(H)-q\big)$. This implies that $K_1\nabla H$ can be decomposed into
at most $\frac{v(H)}2+\frac1g\left(e(H)-\frac{v(H)}2\right)$ edge-disjoint cycles, as $q\le \frac{v(H)}2$.
Combining this with $2e(H)=d\cdot v(H)$, it follows that
\begin{align*}\phi(K_1\nabla H)=s\leq e(H)\left(\frac 1d+\frac1g-\frac1{dg}\right),
\end{align*}
and hence $\frac{e(H)}{s}\geq \frac{dg}{d+g-1}$.

An interesting problem, known as \emph{cage problem}, asks for the construction of regular graphs of specified degree and girth with minimum order.
A graph is called a \emph{$(d,g)$-cage}, if it is $d$-regular with girth $g$ and
it has minimum number of vertices.
Cage problem was first considered by Tutte \cite{TU},
and it is also a recognized unsolved difficult problem.
From the above arguments, the solution to Problem \ref{prob5.4} seems to be related to the cage problem. For more progress on cage problem, one can refer
to the dynamic survey \cite{EX}.


\end{document}